\documentclass[11pt,reqno]{amsart}

\usepackage{amssymb,amsmath,amsthm,amscd,latexsym,amsfonts}
\usepackage{mathtools}
\usepackage[T1]{fontenc}

\usepackage{graphicx}
\usepackage{color}
\usepackage[a4paper,top=3cm, bottom=3cm, left=3cm, right=3cm]{geometry}

\newtheorem{thm}{Theorem}
\newtheorem{defn}{Definition}

\newtheorem{pro}{Proposition}
\newtheorem{rk}{Remark}

\newtheorem{ex}{Example}

\numberwithin{equation}{section} 

\newcommand{\M}{{\mathcal M}}

\newcommand{\bea}{\begin{eqnarray}}
\newcommand{\eea}{\end{eqnarray}}

\newcommand{\R}{\mathbb{R}}

\def\M{\mathcal M}

\def\R{\mathbb{R}}

\begin{document}
\title [Flow of finite-dimensional algebras]
{Flow of finite-dimensional algebras}

\author {M. Ladra, U.A. Rozikov}

\address{M.\ Ladra\\ Department of Algebra, University of Santiago de Compostela, 15782, Spain.}
 \email {manuel.ladra@usc.es}

 \address{U.\ A.\ Rozikov\\ Institute of mathematics, 29, Do'rmon Yo'li str., 100125,
Tashkent, Uzbekistan.} \email {rozikovu@yandex.ru}

\begin{abstract} Each finite-dimensional algebra can be identified 
to the cubic matrix given by structural constants defining the multiplication 
between the basis elements of the algebra.  In this paper we introduce the 
notion of flow (depending on time) of finite-dimensional 
algebras. This flow 
can be considered as a particular case of (continuous-time) dynamical system
whose states are finite-dimensional algebras with matrices of structural constants satisfying an analogue
of Kolmogorov-Chapman equation. These flows of algebras (FA) can also 
be considered as deformations of algebras
with the rule (the evolution equation) given 
by  Kolmogorov-Chapman equation.
We mainly use the multiplications of cubic matrices 
which were introduced by Maksimov and consider Kolmogorov-Chapman 
equation with respect to these multiplications.
If all cubic matrices of structural constants are stochastic (there are several
kinds of stochasticity) then the corresponding FA is called stochastic FA (SFA).
We define SFA generated by
known quadratic stochastic processes.
For some multiplications of cubic matrices we reduce Kolmogorov-Chapman equation 
given for cubic matrices to the equation given for square 
matrices. Using this result many FAs are given  (time
homogenous, time non-homogenous, periodic, etc.).
For a periodic FA we construct a continuum set of finite-dimensional algebras and show that
the corresponding discrete time FA is
dense in the set. Moreover, we give a construction of an FA which contains algebras of arbitrary (finite) dimension.
For several
FAs we describe the time depending behavior (dynamics) of the properties to be baric,
limit algebras, commutative,  evolution algebras or associative algebras.
\end{abstract}

\subjclass[2010] {17D92; 17D99; 60J27}

\keywords{Finite-dimensional algebra; quadratic stochastic process; cubic matrix; time; Kolmogorov-Chapman
equation; baric algebra; property transition; commutative; associative.}

\maketitle

\section{Introduction}

 A dynamical system is a system in which a rule describes the time dependence of the states in the system.
In this paper we consider a dynamical system such that at any given time it
has the state which is an finite-dimensional algebra. The evolution rule of this
dynamical system is given by an evolution (Kolmogorov-Chapman) equation
that describes what future algebras (states) follow from the current algebra.

We note that to each finite-dimensional algebra corresponds a cubic matrix of structural constants.
This cubic matrix generates an evolution quadratic operator.  This is why the
dynamical systems generated by quadratic operators have been a rich
source of analysis for the study of dynamical properties and
modeling in different domains, such as population dynamics
\cite{B,FG,K},  physics \cite{UR}, economy \cite{D} or mathematics \cite{ES,GN,J,ly}.

Etherington introduced the
formal language of abstract algebra to the study of genetics in his
series of seminal papers \cite{e1,e3}. In recent years many
authors have tried to investigate the difficult problem of
classification of these algebras. The most comprehensive references
for the mathematical research done in this area are \cite{ly,m,t,w}.

Recently  in \cite{CLR,Mu,ORT,RM,RM1} some chains of evolution algebras were introduced and studied.
In each of these papers the matrices of structural constants (depending on pair of time $(s, t)$) are square or rectangular and satisfy the
Kolmogorov-Chapman equation. In other words, a chain of evolution algebras is a continuous-time dynamical
system which in any fixed time is an evolution algebra. It is known that if a matrix satisfying
the Kolmogorov-Chapman equation is stochastic, then it generates a Markov process.

In \cite{ht}, the authors investigated time evolution of non-Markov processes as
they occur in coarse-grained description of open and closed systems. Some aspects
of the theory are illustrated for the two-state process and Gauss process.

In this paper we generalize the notion of chain of evolution algebras (given for algebras with \emph{rectangular} matrices)
to a notion of flow of arbitrary finite-dimensional algebras (i.e.
their matrices of structural constants are \emph{cubic} matrices).
Since the matrix of structural constants of each flow of algebras has a
general form, the non-Markov processes of \cite{ht} can be, in particular,
obtained by structural constants matrices
of a chain or a flow of algebras. Thus they can be used in biology and physics.

The paper is organized as follows. In Section~\ref{S:definitions} we give the main definitions and examples related to
quadratic stochastic processes, several kinds of multiplications of cubic matrices, an alternative
definition of quadratic process and flow of algebras (FAs).
We compare the FAs with deformations of algebras. FA is defined by a family (depending on time) of cubic matrices of structural
 constants, which satisfy an analogue of Kolmogorov-Chapman equation. Since there are several types of multiplication of cubic matrices,
one has to fix a multiplication first and then consider Kolmogorov-Chapman equation
with respect to this multiplication. We mainly use the multiplications which were introduced by Maksimov \cite{Mak}.
We define a stochastic FA generated by 
known quadratic stochastic processes.
In Section~\ref{S:reduction} for some multiplications of cubic matrices we reduce Kolmogorov-Chapman
equation given for cubic matrices to an equation given for square 
matrices. Using this result many examples of FAs are given.
In Section~\ref{S:dynamics} for several
FAs we describe time depending behavior of some properties of algebras in FA.

\section{Definitions and examples}\label{S:definitions}
\subsection{Quadratic stochastic operators}
Let $E=\{1,\dots,m\}$ be a finite set.
A \emph{quadratic stochastic operator} (QSO) of a free population \cite{K,ly} is a
(quadratic) mapping from the simplex
\[ S^{m-1}=\{x=(x_1,\dots,x_m)\in \R^m: x_i\geq 0, \ \sum^m_{i=1}x_i=1 \} \]
into itself, of the form
\[
 V: x_k'=\sum^m_{i,j=1}p_{ij,k}x_ix_j, \qquad (k=1,\dots,m),
\]
where $p_{ij,k}$ are the coefficients of heredity and
\begin{equation}\label{as1}
 p_{ij,k}\geq 0, \qquad \sum^m_{k=1}p_{ij,k}=1, \qquad  (i,j,k=1,\dots,m).
\end{equation}
Matrix $(p_{ijk})$ satisfying (\ref{as1}) is called stochastic.
  
 Note that each element $x\in S^{m-1}$ is a probability
distribution on $E$.

\begin{ex}\label{e1} Take $E=\{1,2\}$, $\epsilon\in [0,1/2]$ and consider the QSO, $V \colon S^1\to S^1$, defined as
\begin{align*}
x_1' & =(1-2\epsilon)x_1^2+(1-2\epsilon)x_1x_2,\\
x_2' & =2\epsilon x_1^2+(1+2\epsilon)x_1x_2+x_2^2.
\end{align*}

\end{ex}
In \cite{GMR} (see also \cite{MG}) a review of the theory of QSOs
is given. Note that each quadratic stochastic
operator can be uniquely defined by a stochastic matrix ${\mathbf
P}= \{p_{ij,k}\}^m_{i,j,k=1}$. In \cite{GaR} a constructive
description of ${\mathbf P}$ (i.e. a QSO) is given. This
construction depends on cardinality of $E$, which can be finite or
continual. Some particular cases of this construction were defined
in \cite{GN}.

\subsection{Quadratic stochastic processes}
In this subsection we describe quadratic stochastic processes,
with continuous time, which are related with QSOs
as well as Markov processes with linear operators, this subsection is
based on works \cite{G,MG,S3}.

\begin{defn}\label{d1}  A family $\left\{P^{[s,t]}_{ij,k}: i,j,k\in E, \,
s,t\in \mathbb R_+,  \, t-s\geq 1\right\}$ of functions $P^{[s,t]}_{ij,k}$ with initial state $x^{(0)}\in S^{m-1}$
is called a quadratic stochastic process (QSP) if, for fixed $s,t\in \mathbb R_+$, it satisfies the following conditions:
\begin{itemize}
\item[(i)] $P^{[s,t]}_{ij,k}=P^{[s,t]}_{ji,k}$ for any $i,j,k\in E$;
\item[(ii)] $P^{[s,t]}_{ij,k}\geq 0$ and $\sum_{k\in E}P^{[s,t]}_{ij,k}=1$ for any $i,j,k\in E$;
\item[(iii)] An analogue of the Kolmogorov-Chapman equation: there are two types: for the initial point $x^{(0)}$ and
$s<r<t$ such that $t-r\geq 1$, $r-s\geq 1$ one has
\item [(iii$_A$)]
\[ P^{[s,t]}_{ij,k}=\sum_{m,l\in E}P^{[s,r]}_{ij,m}P^{[r,t]}_{ml,k}x^{(r)}_l,\]
where $x^{(r)}_k$ is defined as
\[x_k^{(r)}=\sum_{i,j\in E}P^{[0,r]}_{ij,k}x_i^{(0)}x_j^{(0)};\]
\item [(iii$_B$)]
\[ P^{[s,t]}_{ij,k}=\sum_{m,l,g,h\in E}P^{[s,r]}_{im,l}P^{[s,r]}_{jg,h}P^{[r,t]}_{lh,k}x^{(s)}_mx^{(s)}_g.\]
\end{itemize}
 \end{defn}
 The QSP is called of type (A) (resp. type (B)) if it satisfies
 the equation (iii$_A$) (resp. (iii$_B$)). The equations (iii$_A$) and (iii$_B$) can be interpreted as different laws of the
 behavior of the `offspring'.

  Here we give some known examples
of different types of QSPs.

\begin{ex}\label{e2} Let $E=\{1,2\}$ and
$(x,1-x)$ be an initial distribution on $E$, $0\leq x\leq 1$.
Consider the QSO of Example~\ref{e1} and the following system of transition probabilities:
\begin{align*}
P^{[s,t]}_{11,1} & =\frac{(1-2\epsilon)^{t-s}}{
2^{t-s-1}}\left[\left(2^{t-s-1}-1\right)(1-2\epsilon)^sx+1\right];\\
P^{[s,t]}_{12,1} & =P^{[s,t]}_{21,1}=\frac{(1-2\epsilon)^{t-s}}{
2^{t-s-1}}\left[\left(2^{t-s-1}-1\right)(1-2\epsilon)^sx+\frac{1}{2}\right];\\
P^{[s,t]}_{22,1} & =\frac{(1-2\epsilon)^{t-s}}{
2^{t-s-1}}\left(2^{t-s-1}-1\right)x;\\
P^{[s,t]}_{ij,2} & =1-P^{[s,t]}_{ij,1},  \qquad \qquad \qquad i,j=1,2.
\end{align*}
For $\epsilon\in [0,1/2]$ these transition probabilities generate
a QSP which is of type (A) and type (B), simultaneously. In this
case we have
\[  x_1^{(t)}=(1-2\epsilon)^tx, \qquad x_2^{(t)}=1-(1-2\epsilon)^tx.\]
\end{ex}

\begin{ex}\label{e3} Let $E=\{1,2\}$ and
$(x,1-x)$ be an initial distribution on $E$, $0\leq x\leq 1$.
Consider the following system of transition functions:
\begin{align*}
P^{[s,t]}_{11,1} & =\frac{\epsilon^{t-s}}{
2^{t-s-1}}\cdot \frac{s+1}{t+1}\cdot \left[\left(2^{t-s-1}-1\right)\cdot \frac{\epsilon^s}{s+1}x+1\right];\\
P^{[s,t]}_{12,1} & =P^{[s,t]}_{21,1}=\frac{\epsilon^{t-s}}{
2^{t-s-1}}\cdot \frac{s+1}{t+1}\cdot\left[\left(2^{t-s-1}-1\right)\cdot \frac{\epsilon^s}{s+1}x+\frac{1}{2}\right];\\
P^{[s,t]}_{22,1} & =\frac{2^{t-s-1}-1}{
2^{t-s-1}}\cdot \frac{\epsilon^t}{t+1}x;\\
P^{[s,t]}_{ij,2} & =1-P^{[s,t]}_{ij,1}, \qquad \qquad \qquad i,j=1,2.
\end{align*}
These functions generate a QSP of type (A) and (B) for any $\epsilon\in [0,1]$. Moreover,
\[ x_1^{(t)}=\frac{\epsilon^t}{t+1}x, \qquad  x_2^{(t)}=1-\frac{\epsilon^t}{t+1}x, \qquad t\geq 1.\]
\end{ex}

\begin{ex}\label{e4}  Let $E=\{1,2,3\}$ and
$(x_1,x_2,1-x_1-x_2)$ be an initial distribution on $E$, where
$x_1\geq 0$, $x_2\geq 0$, $x_1+x_2\leq 1$. Consider the following system
of transition probabilities:
\begin{align*}
P^{[s,t]}_{11,1} & = 2\epsilon^{t-s}+\frac{2^{t-s-1}-1}{
2^{t-s-1}}x_1^{(t+1)};\\
P^{[s,t]}_{12,1}& =P^{[s,t]}_{13,1}=\epsilon^{t-s}+\frac{2^{t-s-1}-1}{
2^{t-s-1}}x_1^{(t+1)};\\
P^{[s,t]}_{22,1}& =P^{[s,t]}_{23,1}=P^{[s,t]}_{33,1}=\frac{2^{t-s-1}-1}{
2^{t-s-1}}x_1^{(t+1)};\\
P^{[s,t]}_{11,2} & =P^{[s,t]}_{13,2}=P^{[s,t]}_{33,2}=\frac{2^{t-s-1}-1}{
2^{t-s-1}}x_2^{(t+1)};\\
P^{[s,t]}_{22,2} & =2\epsilon^{t-s}+\frac{2^{t-s-1}-1}{
2^{t-s-1}}x_2^{(t+1)};\\
P^{[s,t]}_{12,2}& =P^{[s,t]}_{23,2}=\epsilon^{t-s}+\frac{2^{t-s-1}-1}{
2^{t-s-1}}x_2^{(t+1)};\\
P^{[s,t]}_{ij,3}& =1-P^{[s,t]}_{ij,1}-P^{[s,t]}_{ij,2}, \qquad \qquad \qquad i,j=1,2,3,
\end{align*}
where $x_1^{(t)}=(2\epsilon)^tx_1,  \ \
x_2^{(t)}=(2\epsilon)^tx_2$. For $\epsilon\in [0,1/2]$, these
transition probabilities generate a QSP which is of type (A), but
is not of type (B).
\end{ex}
 \begin{ex}\label{e5} Let $E=\{1,2,3\}$ and
$(x_1,x_2,1-x_1-x_2)$ be an initial distribution on $E$, where
$x_1\geq 0, \, x_2\geq 0, \, x_1+x_2\leq 1$. Consider the following system
of transition probabilities:
\begin{align*}
\widetilde{P}^{[s,t]}_{11,1} & =2\epsilon^{t-s}+\frac{2^{t-s-1}-1}{
2^{t-s-1}}x_1^{(t)}; \\
\widetilde{P}^{[s,t]}_{12,1} & =\widetilde{P}^{[s,t]}_{13,1}=\epsilon^{t-s}+\frac{2^{t-s-1}-1}{
2^{t-s-1}}x_1^{(t)};\\
\widetilde{P}^{[s,t]}_{22,1} & =\widetilde{P}^{[s,t]}_{23,1}=\widetilde{P}^{[s,t]}_{33,1}=\frac{2^{t-s-1}-1}{
2^{t-s-1}}x_1^{(t)};\\
\widetilde{P}^{[s,t]}_{11,2} & =\widetilde{P}^{[s,t]}_{13,2}=\widetilde{P}^{[s,t]}_{33,2}=\frac{2^{t-s-1}-1}{
2^{t-s-1}}x_2^{(t)};\\
\widetilde{P}^{[s,t]}_{22,2} & =2\epsilon^{t-s}+\frac{2^{t-s-1}-1}{
2^{t-s-1}}x_2^{(t)};\\
\widetilde{P}^{[s,t]}_{12,2}& =\widetilde{P}^{[s,t]}_{23,2}=\epsilon^{t-s}+\frac{2^{t-s-1}-1}{
2^{t-s-1}}x_2^{(t)};\\
\widetilde{P}^{[s,t]}_{ij,3} & =1-\widetilde{P}^{[s,t]}_{ij,1}-\widetilde{P}^{[s,t]}_{ij,2}, \qquad \qquad \qquad  i,j=1,2,3,
\end{align*}
where $x_1^{(t)}=(2\epsilon)^tx_1,  \ \
x_2^{(t)}=(2\epsilon)^tx_2$. For $\epsilon\in [0,1/2]$ these
transition probabilities generate a QSP which is of type (B), but
is not of type (A).
\end{ex}

\subsection{Algebras and Cubic matrices}

Let $K$ be a field, and let $\mathcal A$ be a vector space over $K$ equipped with an additional
binary operation from $\mathcal A\times \mathcal A$ to $\mathcal A$, denoted
by $\cdot$ (we  will simply write $xy$ instead $x\cdot y$).  Then $\mathcal A$ is an \emph{algebra} over $K$ if the following identities
hold for every elements $x, y$, and $z$ of $\mathcal A$, and every elements $a$ and $b$ of $K$:
 $(x + y)z = xz + yz$, $x(y + z) = xy + xz$, $(ax)(by) = (ab) (xy)$.

For algebras over a field, the bilinear multiplication from $\mathcal A\times \mathcal A$ to $\mathcal A$,
is completely determined by the multiplication of the basis elements of $\mathcal A$. Conversely, once a basis for $\mathcal A$
 has been chosen, the products of the basis elements can be set arbitrarily, and then extended in a unique
 way to a bilinear operator on $\mathcal A$ so that the resulting multiplication satisfies the algebra laws.

Given a field $K$, any finite-dimensional algebra can be specified up
to isomorphism by giving its dimension (say $m$), and specifying $m^3$ structure constants $c_{ijk}$, which are scalars.
These structure constants determine the multiplication in $\mathcal A$ via the following rule:
\[{e}_{i} {e}_{j} = \sum_{k=1}^m c_{ijk} {e}_{k},\]
where $e_1,\dots,e_m$ form a basis of $\mathcal A$.

Thus the multiplication of a finite-dimensional algebra is given by a cubic matrix $(c_{ijk})$.
Following \cite{Mak} (see also \cite{Pa})  recall the notion of cubic matrix and different associative multiplication rules of cubic matrices:
a cubic matrix $Q=(q_{ijk})_{i,j,k=1}^m$ is a $m^3$-dimensional vector which can be uniquely written as
\[Q=\sum_{i,j,k=1}^m   q_{ijk} (i,j,k),\]
where $(i,j,k)$ denotes the cubic unit (basis) matrix, i.e. $(i,j,k)$ is a $m^3$ cubic matrix whose
$(i,j,k)$th entry is equal to 1 and all other entries are equal to 0.
Define the following multiplications for such basis matrices:
\begin{itemize}
  \item[(C)]
  \[(i,j,k)\cdot (l,n,r)=
  \begin{cases}
  (i,j,r), & \text{if} \ \ k=l, \ \ \text{and} \ \  j=n,\\
  \quad 0, & \text{otherwise}.
  \end{cases} \]
Extending this multiplication to arbitrary cubic matrices
$$A=(a_{ijk})_{i,j,k=1}^m, \ \ B=(b_{ijk})_{i,j,k=1}^m, \ \
C=(c_{ijk})_{i,j,k=1}^m,$$ we get that the entries of $C=AB$ can be written as
  \begin{equation}\label{C}
  c_{ijr}=\sum_{k=1}^ma_{ijk}b_{kjr}.
  \end{equation}
   \item[(D)]
    \[(i,j,k)\cdot (l,n,r)=
    \begin{cases}
  (i,j,r), & \text{if} \ \ k=l,\\
  \quad 0, & \text{otherwise}.
  \end{cases}\]
In this case
  \begin{equation}\label{D}
  c_{ijr}=\sum_{k,n=1}^ma_{ijk}b_{knr}.
  \end{equation}
   \item[(E)]
   \[(i,j,k)\cdot (l,n,r)=
    \begin{cases}
  (i,n,r),& \text{if} \ \ k=l,\\
  \quad 0, & \text{otherwise}.
 \end{cases}\]
Consequently
  \[
  c_{inr}=\sum_{j,k=1}^ma_{ijk}b_{knr}.
  \]
\end{itemize}
\begin{rk}\label{r1} In \cite{Mak} there is more general multiplication:
 \[(i,j,k)\cdot (l,n,r)=
 \begin{cases}
  (i,j\circ n,r), & \text{if} \ \ k=l,\\
  \qquad 0, & \text{otherwise},
  \end{cases}
 \]
 where $\circ$ is an arbitrary associative operation on $E=\{1,2,\dots,m\}$.
 In case $j\circ n=j$ (resp. $=n$) this multiplication coincides with (D) (resp. (E)).
 Recently in \cite{Ba} the authors provided 15 associative multiplication rules of cubic matrices,
 some of them coincide with Maksimov's multiplication rules.
 In \cite{Pa} there are other kinds of multiplications for cubic stochastic matrices.
 To avoid many cases, in this paper we consider only multiplications (C), (D), (E).
 \end{rk}
\subsection{Flow of algebras}

Consider a family $\left\{\mathcal A^{[s,t]}:\ s,t \in \R,\ 0\leq s\leq t
\right\}$ of arbitrary $m$-dimensional algebras over the field $\R$,
with basis $e_1,\dots,e_m$ and multiplication table
\begin{equation}\label{1}
 e_ie_j =\sum_{k=1}^mc_{ijk}^{[s,t]}e_k, \ \ i,j=1,\dots,m.\end{equation}
 Here parameters $s,t$ are
considered as time.

Denote by
$\M^{[s,t]}=\left(c_{ijk}^{[s,t]}\right)_{i,j,k=1,\dots,m}$ the matrix
of structural constants of $\mathcal A^{[s,t]}$.

\begin{defn}\label{d2}
Define the following flow of algebras (FA):
\begin{itemize}
\item[$\bullet$] A family $\left\{\mathcal A^{[s,t]}:\ s,t \in \R,\ 0\leq s\leq t
\right\}$ of $m$-dimensional algebras over the field $\R$
is called a stochastic flow of algebras (SFA) if the matrices
$\M^{[s,t]}$ of structural constants satisfy the conditions
of Definition~\ref{d1}. The SFA is called of type (A) (resp. type (B)) if
it corresponds to a QSP of type (A) (resp. (B)).

\item[$\bullet$] A family $\left\{\mathcal A^{[s,t]}:\ s,t \in \R,\ 0\leq s\leq t
\right\}$ of $m$-dimensional algebras over the field $\R$
is called an FA if the matrices
$\M^{[s,t]}$ of structural constants satisfy the
Kolmogorov-Chapman equation (for cubic matrices):
\begin{equation}\label{KC}
\M^{[s,t]}=\M^{[s,\tau]}\M^{[\tau,t]}, \qquad \text{for all} \ \ 0\leq s<\tau<t.
\end{equation}
The FA is called of type (C) (resp. type (D), (E)) if
the equation \eqref{KC} is considered with respect to multiplication of type (C) (resp. (D), (E)) mentioned above.
\end{itemize}
\end{defn}
Let us give following useful remarks.
\begin{rk} We note that
\begin{itemize}
\item[$\bullet$] SFA corresponds to a family of stochastic matrices (i.e. satisfy condition {\rm (ii)} of Definition~\ref{d1}).
But in general for  FAs we do not assume that the corresponding matrices are stochastic, they only satisfy \eqref{KC}.
\item[$\bullet$] Since there are several different kinds of multiplications of cubic matrices (see Remark~\ref{r1}),
 Definition~\ref{d2} gives several types of FA.  To construct an FA one has to fix a multiplication rule first and then consider
 equation \eqref{KC} with respect to the fixed multiplication.
\end{itemize}
\end{rk}
\begin{rk} {\rm Interpretations:} The multiplication $e_ie_j$ given by the equality \eqref{1}
can be considered as action of the particle $e_i$ on the particle $e_j$ at time $s$, as an interaction process, then
with rate $c_{ijk}^{[s,t]}$ a particle $e_k$ appears at time $t$. Kolmogorov-Chapman equation \eqref{KC} gives 
time depending evolution law of the interacting process (dynamical system).
\end{rk}

\begin{rk} {\rm Comparison with the deformations of algebras:} In case of finite-dimensional
algebras a deformation is the transformation of a
given algebra to another algebra (with the same dimension) \cite{Ger}. This is equivalent to change (transform)
a given multiplication table (structural constants) to another one. Often, the algebra at $t = 0$ is considered as
an initial algebra and the algebra at the current time $t$ as the current algebra.
In \cite{Ger} deformations of an algebra $A$ are given by a bilinear function
$f_t$, i.e., one considers the algebra $A_t$ with
 multiplication $f_t$ as the generic element of a ``one-parameter family of
 deformations of $A$''. Thus deformations of an algebra
 are given by the rule $f_t$ (which has an explicit form), similarly a flow of algebras is given by $\mathcal M^{[s,t]}$ with the rule \eqref{KC}.
Some relations of an FA can be also seen in the following example: a plane deformation \cite{O} is restricted to the plane described by the basis vectors $e_1$, $e_2$.
It is known that the deformation gradient of this plane deformation has the form
\[ \mathbf{F}=\begin{pmatrix}
\cos\theta & \sin\theta & 0 \\
-\sin\theta & \cos\theta & 0 \\
 0 & 0 & 1
 \end{pmatrix}
    \begin{pmatrix}
    \lambda_1 & 0 & 0 \\[2mm]
     0 & \lambda_2 & 0 \\[2mm]
     0 & 0 & 1 \end{pmatrix},
\]
where $\theta$ is the angle of rotation and $\lambda_1, \lambda_2$ are the principal stretches.
Comparing this matrix with matrix \eqref{defo} (see below) one can see that
the corresponding FA and the plane deformation have similar matrices.
\end{rk}
{\bf Alternatives to Definition~\ref{d1} of QSP:} Following \cite{Mak} one can define several kinds of stochastic cubic  matrices:
a cubic matrix $P=(p_{ijk})_{i,j,k=1}^m$ is called stochastic of type $(1,2)$ if
\[p_{ijk}\geq 0, \qquad  \sum_{i,j=1}^mp_{ijk}=1, \ \ \text{for all} \ k.\]
The cubic stochastic matrices of type $(1,3)$ and $(2,3)$ can be defined similarly.  Moreover
a cubic matrix can be called stochastic\footnote{This definition of stochaticity has been already used in (\ref{as1}).} if
\[p_{ijk}\geq 0, \qquad  \sum_{k=1}^mp_{ijk}=1, \ \ \text{for all} \  i,j.\]
The last one can be also given with respect to first and second index.
Maksimov \cite{Mak} also defined twice stochastic matrix:
a (2,3) stochastic cubic matrix is called twice stochastic if
\[\sum_{i=1}^m p_{ijk}=\frac{1}{m}, \qquad \text{for all} \ j,k.\]
We denote by $\mathcal S$ the set of all these definitions of stochasticity and
denote by $\mathbb M$ the set
of all known multiplication rules of cubic matrices.
Now we give an alternative definition of QSP:
\begin{defn}\label{d1a}
A family $\{P_{ijk}^{[s,t]}: i,j,k=1,\dots, m; \ s,t\in\R_+\}$ is called a QSP of type $(\sigma,\mu)$ if for each time $s$ and $t$ the cubic matrix
$\left(P_{ijk}^{[s,t]}\right)$ is stochastic in sense $\sigma\in\mathcal S$ and satisfies Kolmogorov-Chapman equation \eqref{KC} with
respect to the multiplication $\mu\in \mathbb M$.

\end{defn}
\begin{rk} Different multiplication rules introduced by Maksimov were aimed to ensure
that stochastic nature of the cubic matrix
product remains when multiplying them.
Fix a pair $(\sigma,\mu)\in \mathcal S\times\mathbb M$ then Definition 3 makes sense for this
choice  $(\sigma,\mu)$ if there exists at least one family
$\left(P_{ijk}^{[s,t]}\right)$ of cubic matrices
which are stochastic in sense $\sigma$ and satisfies Kolmogorov-Chapman equation with respect to multiplication $\mu$.
Denote by $\mathbb Q(\sigma,\mu)$ the set of all QSP of type $(\sigma,\mu)$.
An interesting problem is to find pairs $(\sigma,\mu)$ with property $\mathbb Q(\sigma,\mu)=\emptyset$. 
This problem will be considered in a separate work.\footnote{This remark is added corresponding to a suggestion of reviewer.}
\end{rk}

Comparing with Definition~\ref{d1} we note that in this definition of QSP we
do not assume $t-s\geq 1$ and it does not
depend on an initial point $x^{(0)}$.
For the theory of QSPs in the sense of Definition~\ref{d1} see \cite{MG}.
Such theory should be developed for QSPs defined in Definition~\ref{d1a}.

\begin{defn}
 An FA is called a (time) homogenous FA if
the matrix $\M^{[s,t]}$ depends only on $t-s$. In this case we write
$\M^{[t-s]}$.
\end{defn}
Now we shall give several concrete examples of FAs.

First we note that all above-mentioned examples of QSP generate
SFAs. 

The following is a new one:

\begin{ex}\label{e6} Let $E=\{1,2,\dots,m\}$. Consider a family of $m$-dimensional stochastic vectors:
 $a(t)=(a_1(t), a_2(t), \dots, a_m(t))$,
i.e. $a_i(t)\geq 0$, $\sum_i a_i(t)=1$ for any $t\geq 0$. For each
pair $s,t$ we define a stochastic matrix
$Q^{[s,t]}=(q_{il}^{[s,t]})_{i,l\in E}$, where
$q_{il}^{[s,t]}=a_l(t)$ for all $i\in E$, i.e. it does not depend
on $s$.  It is easy to see that this matrix satisfies the
Kolmogorov-Chapman equation:
\[Q^{[s,t]}=Q^{[s,\tau]}Q^{[\tau,t]}, \qquad \text{for all} \ \ 0\leq s<\tau<t.\]
Now define functions
\[P^{[s,t]}_{ij,k}=q_{ik}^{[s,t]}=a_k(t).\]
One can check that the defined family $\{P^{[s,t]}_{ij,k}\}$ is
a QSP. Thus it generates a SFA.
\end{ex}

Now we take $m=2$, i.e. $E=\{1,2\}$ and construct examples of a flow of two-dimensional algebras of type (C).
In this case  equation \eqref{KC} has the following form:
\[
c_{ijr}^{[s,t]}=c_{ij1}^{[s,\tau]}c_{1jr}^{[\tau,t]}+c_{ij2}^{[s,\tau]}c_{2jr}^{[\tau,t]}, \qquad i,j,r=1,2.
\]
This is a quadratic system of eight equations with eight unknown functions $c_{ijr}^{[s,t]}$ of two variables
$s,t$, $0\leq s<t$. It is easy to see that the equations for $j=1$ and $j=2$ are independent. Therefore it suffices to solve the system only for $j=1$.
Denoting
$a_{ir}^{[s,t]}=c_{i1r}^{[s,t]}$ we get
\begin{align}\label{fe}
a_{11}^{[s,t]} & =a_{11}^{[s,\tau]}a_{11}^{[\tau,t]}+a_{12}^{[s,\tau]}a_{21}^{[\tau,t]}, \notag\\
a_{12}^{[s,t]} & =a_{11}^{[s,\tau]}a_{12}^{[\tau,t]}+a_{12}^{[s,\tau]}a_{22}^{[\tau,t]},\\
a_{21}^{[s,t]} & =a_{21}^{[s,\tau]}a_{11}^{[\tau,t]}+a_{22}^{[s,\tau]}a_{21}^{[\tau,t]},\notag\\
a_{22}^{[s,t]} & =a_{21}^{[s,\tau]}a_{12}^{[\tau,t]}+a_{22}^{[s,\tau]}a_{22}^{[\tau,t]}.\notag
 \end{align}

The full set of solutions to the system \eqref{fe} is not known yet. But there is a very wide class of its solutions
see \cite{CLR,ORT,RM}. One of these known solutions is the following (non-stochastic, time-homogeneous) matrix:
\begin{equation}\label{defo}
\begin{pmatrix}
a_{11}^{[s,t]}& a_{12}^{[s,t]}\\
a_{21}^{[s,t]}& a_{22}^{[s,t]}
\end{pmatrix}=\begin{pmatrix}
\cos (t-s)& \sin (t-s)\\
-\sin (t-s)& \cos (t-s)
\end{pmatrix}.
\end{equation}
Thus the following is a wide class of examples:
\begin{ex}\label{e7}
 Any pair $(a_{ij}^{[s,t]}), (\bar a_{ij}^{[s,t]})$ of solutions of the system \eqref{fe}
generates an FA of type (C) corresponding to the matrix $\mathcal M^{[s,t]}$ with entries
\[  c_{i1r}^{[s,t]}=a_{ir}^{[s,t]}, \qquad  c_{i2r}^{[s,t]}=\bar a_{ir}^{[s,t]}.\]
\end{ex}

In the next section we shall give general arguments to solve the equation \eqref{KC} and after that we shall give
examples of FAs of type (D) and (E).

Note that for each fixed $s,t$ the matrix $\M^{[s,t]}$ of a SFA defines an algebra $\mathcal A^{[s,t]}$ over $\R$.
This algebra can be considered as an \emph{evolution algebra of a free population} (see \cite[page 15  and Chapter 3]{ly}), but Lyubich 
used for this algebra the notion {\it evolutionary algebra}.

A \emph{character} for an algebra $A$ is a nonzero multiplicative
linear form on $A$, that is, a nonzero algebra homomorphism from $A$
to $\R$ \cite{ly}. Not every algebra admits a character. For
example, an algebra with the zero multiplication has no character.

\begin{defn}
 A pair $(A, \chi)$ consisting of an algebra $A$ and a
character $\chi$ on $A$ is called a \emph{baric algebra}. The
homomorphism $\chi$ is called the weight (or baric) function of
$A$ and $\chi(x)$ the weight (baric value) of $x$.
\end{defn}
In \cite{ly} for the evolution algebra of a free population it is
proven that there is a character $\chi(x)=\sum_i x_i$, therefore
that algebra is baric. Moreover, such an algebra has the following properties:
commutative, but non-associative in general.

The following proposition is known (see \cite{w}):
\begin{pro}\label{pb}
A finite-dimensional commutative real algebra $A$ is a baric
algebra if and only if $A$ has a basis $e_1,\dots , e_m$ such that the constants $c_{ijk}$ defined by
\[{e}_{i} {e}_{j} = \sum_{k=1}^m c_{ijk} {e}_{k}\]
satisfy the relation
\[\sum_{k=1}^m c_{ijk}=1.\]
\end{pro}

In the next sections we shall study the dynamics (depending on time) of properties of algebras in FAs.

\section{Reduction of Kolmogorov-Chapman's equations from cubic matrices to square ones} \label{S:reduction}
In this section for each multiplication (C), (D), (E) we show that equation
\eqref{KC} written for cubic matrices (not necessarily stochastic) can be reduced to
similar equations for square matrices (different for different multiplications).

\emph{Case (C)}. Let $\mathcal M_j^{[s,t]}=(c^{[s,t]}_{ijk})_{i,k=1}^m$ be the $j$th layer of the matrix $\mathcal M^{[s,t]}$.
  The following proposition is very useful:
\begin{pro} Any solution of the equation \eqref{KC} for multiplication of type (C) is
a direct sum of solutions of the following $m$ independent equations:
\[
\M_j^{[s,t]}=\M_j^{[s,\tau]}\M_j^{[\tau,t]}, \quad \text{for all} \quad 0\leq s<\tau<t, \quad j=1,\dots,m.
\]
\end{pro}
\begin{proof} This is consequence of the equality \eqref{C}, which shows that layers are independent in the multiplication of type (C).
\end{proof}

\emph{Case (D).} For the cubic matrix $\mathcal M^{[s,t]}$ consider the square matrix  $\overline{\mathcal M}^{[s,t]}=(\bar c^{[s,t]}_{ik})$ with
\begin{equation}\label{cd}
\bar c^{[s,t]}_{ik}=\sum_{j=1}^m  c_{ijk}^{[s,t]}, \qquad  i,k=1,\dots,m.
\end{equation}

\begin{pro}\label{pD} Any solution of equation \eqref{KC} for the multiplication of type (D) can be given
by a solution of the system \eqref{cd} with a matrix $\overline{\mathcal M}^{[s,t]}=(\bar c^{[s,t]}_{ik})$ which satisfies
\begin{equation}\label{KCD}
\overline\M^{[s,t]}=\overline\M^{[s,\tau]} \ \overline\M^{[\tau,t]}, \ \ \text{for all} \quad 0\leq s<\tau<t.
\end{equation}
\end{pro}
\begin{proof}
 Using the multiplication rule \eqref{D}, we write the equation \eqref{KC} in the following form:
\begin{equation}\label{y}
c_{ijr}^{[s,t]}=\sum_{k=1}^m \left(c^{[s,\tau]}_{ijk}\sum_{n=1}^m c^{[\tau,t]}_{knr}\right)=
\sum_{k=1}^m c^{[s,\tau]}_{ijk}\bar c^{[\tau,t]}_{kr}.\end{equation}
Taking sum over $j$, from the last equality we get
 \[\bar c_{ir}^{[s,t]}=\sum_{k=1}^m \bar c^{[s,\tau]}_{ik}\bar c^{[\tau,t]}_{kr},\]
 this is the equality \eqref{KCD}.
\end{proof}
To illustrate Proposition~\ref{pD} we give the following example:
\begin{ex}\label{e8}
Take $m=2$ and an arbitrary function $\Psi(t)$, $t\geq 0$, with $\Psi(t)\ne 0$.
Consider the following square matrix:
\begin{equation}\label{ut}
\overline \M^{[s,t]}=\frac{1}{2} \begin{pmatrix*}[r]
\frac{\Psi(t)}{\Psi(s)} & -\frac{\Psi(t)}{\Psi(s)}\\[2mm]
-\frac{\Psi(t)}{\Psi(s)} & \frac{\Psi(t)}{\Psi(s)}
\end{pmatrix*}.
\end{equation}
It is easy to check that this matrix satisfies the equality \eqref{KCD}.
Write cubic matrix $\M^{[s,t]}$ for $m=2$ in the following convenient form:
\begin{equation}\label{edd}
\M^{[s,t]}= \begin{pmatrix}
c_{111}^{[s,t]} &c_{112}^{[s,t]}& \vline &c_{211}^{[s,t]} &c_{212}^{[s,t]}\\[3mm]
c_{121}^{[s,t]} &c_{122}^{[s,t]}& \vline &c_{221}^{[s,t]} &c_{222}^{[s,t]}
\end{pmatrix}.
\end{equation}
Using Proposition~\ref{pD} we shall construct $\M^{[s,t]}$ which corresponds to the taken matrix \eqref{ut}:
from \eqref{KC} (i.e. \eqref{y}) we get
\begin{align*}
c_{ij1}^{[s,t]} & = \  \ \frac{\Psi(t)}{2\Psi(\tau)}c_{ij1}^{[s,\tau]}-\frac{\Psi(t)}{ 2\Psi(\tau)}c_{ij2}^{[s,\tau]},   \qquad  i,j=1,2, \\
c_{ij2}^{[s,t]} & =-\frac{\Psi(t)}{ 2\Psi(\tau)}c_{ij1}^{[s,\tau]}+\frac{\Psi(t)}{2\Psi(\tau)}c_{ij2}^{[s,\tau]}, \qquad i,j=1,2.
\end{align*}

Consequently $c_{ij1}^{[s,t]}=-c_{ij2}^{[s,t]}$. Therefore
\[c_{ij1}^{[s,t]}=\frac{\Psi(t)}{\Psi(\tau)}c_{ij1}^{[s,\tau]} \Leftrightarrow \frac{c_{ij1}^{[s,t]}}{ \Psi(t)}=\frac{c_{ij1}^{[s,\tau]}}{\Psi(\tau)}.\]
It follows from the last equality  that $\frac{c_{ij1}^{[s,t]}}{\Psi(t)}$ does not depend on $t$, i.e. there exists a function $\kappa_{ij}(s)$
such that $\frac{c_{ij1}^{[s,t]}} {\Psi(t)}=\kappa_{ij}(s)$. Thus
\begin{equation}\label{ecc}
 c_{ij1}^{[s,t]}=\kappa_{ij}(s)\Psi(t), \qquad  c_{ij2}^{[s,t]}=-\kappa_{ij}(s)\Psi(t).
\end{equation}
By \eqref{cd} for \eqref{ut} we should have
\[c_{111}^{[s,t]}+c_{121}^{[s,t]}=-(c_{112}^{[s,t]}+c_{122}^{[s,t]})=-(c_{211}^{[s,t]}+c_{221}^{[s,t]})=c_{212}^{[s,t]}+c_{222}^{[s,t]}=\frac{\Psi(t)}{ 2\Psi(s)}.\]
By using \eqref{ecc}, we get from these equalities  that
\[\kappa_{11}(s)+\kappa_{12}(s)=-(\kappa_{21}(s)+\kappa_{22}(s))=\frac{1}{2\Psi(s)}.\]
Consequently the matrix \eqref{edd} has the following form:
\begin{equation}\label{ed1}
\M^{[s,t]}=\Psi(t) \begin{pmatrix}
\kappa_{11}(s) &-\kappa_{11}(s)& \vline &\kappa_{21}(s) &-\kappa_{21}(s) \\[3mm]
\frac{1}{2\Psi(s)}-\kappa_{11}(s) &\kappa_{11}(s)-\frac{1}{2\Psi(s)}& \vline &-\frac{1}{2\Psi(s)}-\kappa_{21}(s) &\kappa_{21}(s)+\frac{1}{2\Psi(s)}
\end{pmatrix},
\end{equation}
where $\Psi\ne 0$, $\kappa_{11}$ and $\kappa_{21}$ are arbitrary functions of $t\geq 0$.
 This is a $2\times2\times 2$ cubic matrix which satisfies \eqref{KC} and has three parameter functions.
Thus we proved the following:
\begin{pro} For any fixed functions $\Psi\ne 0$, $\kappa_{11}$ and $\kappa_{21}$ the matrix \eqref{ed1} generates an FA of type (D).
\end{pro}
\end{ex}
\begin{defn}
An FA is called periodic if its matrix $\M^{[s,t]}$
is periodic with respect to at least one of the variables  $s$, $t$,
i.e. (periodicity with respect to $t$) $\M^{[s,t+P]}=\M^{[s,t]}$ for
all values of $t$. The constant $P$ is called the period, and is
required to be nonzero.
\end{defn}
Now we shall construct a periodic FA.
\begin{ex}\label{e9}
Again consider $m=2$ and using Proposition~\ref{pD} we shall construct $\M^{[s,t]}$ corresponding to the matrix \eqref{defo}:
from \eqref{KC} we get
\begin{align*}
c_{ij1}^{[s,t]} & =\cos(t-\tau)c_{ij1}^{[s,\tau]}-\sin(t-\tau)c_{ij2}^{[s,\tau]}, \qquad i,j=1,2, \\
c_{ij2}^{[s,t]} & =\sin(t-\tau)c_{ij1}^{[s,\tau]}+\cos(t-\tau)c_{ij2}^{[s,\tau]}, \qquad  i,j=1,2.
\end{align*}

It is easy to check that this system has a solution in the form
\begin{equation}\label{cos}
c_{ij1}^{[s,t]}=A_{ij}\cos(t-s)-B_{ij}\sin(t-s), \quad c_{ij2}^{[s,t]}=B_{ij}\cos(t-s)+A_{ij}\sin(t-s),
\end{equation}
where $A_{ij}$ and $B_{ij}$ are arbitrary numbers.

 By \eqref{cd} for \eqref{defo} we should have
\begin{align*}
c_{111}^{[s,t]}+c_{121}^{[s,t]}&=\cos(t-s), \qquad  \ \  c_{112}^{[s,t]}+c_{122}^{[s,t]}=\sin(t-s),\\
c_{211}^{[s,t]}+c_{221}^{[s,t]}&=-\sin(t-s), \qquad c_{212}^{[s,t]}+c_{222}^{[s,t]}=\cos(t-s).
\end{align*}
 By using \eqref{cos},  we get from these equalities that
\begin{align*}
(A_{11}+A_{12})\cos(t-s)-(B_{11}+B_{12})\sin(t-s) & =\cos(t-s),\\
(B_{11}+B_{12})\cos(t-s)+(A_{11}+A_{12})\sin(t-s) & =\sin(t-s),\\
(A_{21}+A_{22})\cos(t-s)-(B_{21}+B_{22})\sin(t-s) & =-\sin(t-s),\\
(B_{21}+B_{22})\cos(t-s)+(A_{21}+A_{22})\sin(t-s) & =\cos(t-s).
\end{align*}
Consequently
\[ A_{12}=1-A_{11}, \quad B_{12}=-B_{11}, \quad A_{22}=-A_{21}, \quad B_{22}=1-B_{21}. \]
Now denoting $a=A_{11}$, $b=B_{11}$, $c=A_{21}$, $d=B_{21}$ and $\hat t=t-s$, we get the cubic matrix \eqref{edd} as
\begin{multline}\label{ed3}
\M^{[s,t]}=\M^{[\hat t]}=
\left(\begin{array}{cccc}
a\cos\hat t-b\sin\hat t & \ \ b\cos\hat t+a\sin\hat t\\[3mm]
(1-a)\cos\hat t+b\sin\hat t & \ \quad -b\cos\hat t+(1-a)\sin\hat t
\end{array}\right. \\
\left.\begin{array}{ccccc}
\vline& c\cos\hat t-d\sin\hat t & \ \ d\cos\hat t+c\sin\hat t \\[3mm]
\vline& -c\cos\hat t-(1-d)\sin\hat t & \ \quad (1-d)\cos\hat t-c\sin\hat t
\end{array}\right),
\end{multline}
this is a $2\times2\times 2$ cubic matrix which satisfies \eqref{KC} and has four arbitrary parameters $a,b,c,d\in \mathbb R$.
Thus we proved the following:
\begin{pro} For any fixed parameters $a,b,c,d\in \mathbb R$  the matrix \eqref{ed3} generates a homogeneous, periodic FA of type (D).
\end{pro}
\end{ex}
{\bf A $m$-dimensional time non-homogenous FA of type (D).}
For arbitrary $m$ we shall give an example of time
non-homogenous FA of type (D),  i.e. we shall proof the following:

\begin{thm}\label{tA}
 Let $\{A^{[t]}=(a_{ij}^{[t]}), \,t\geq 0\}$ be a family of
invertible (for all $t$), $m\times m$ square matrices\footnote{Construction of a family of invertible $m\times m$ matrices $A^{[t]}$
is not difficult, for example, one can take a lower or an upper triangular matrix with non-zero entries in the diagonal.
 Thus this theorem describes a very rich class of FAs.} and let $(A^{[t]})^{-1}=(b_{ij}^{[t]})$ denote the inverse of $A^{[t]}$.
Assume that $\beta_{ijk}^{(s)}$, $i,j,k=1, \dots,m$, are arbitrary functions such that
\[\sum_{j=1}^m\beta_{ijk}^{(s)}=a_{ik}^{[s]}, \qquad \text{for any} \ \ i,k \ \ \text{and} \ \ s.\]
Then cubic matrix
\[\M^{[s,t]}=\left(\sum_{k=1}^m\beta_{ijk}^{(s)}b_{kr}^{[t]}\right)_{i,j,r=1}^m\]
generates an FA of type (D).
\end{thm}
\begin{proof}
 Define the following
matrix:
\[ \overline\M^{[s,t]}=A^{[s]}(A^{[t]})^{-1}.\]
It is known (see \cite{CLR}) that this matrix satisfies \eqref{KCD}.
We shall use it to construct a cubic matrix which satisfies \eqref{KC}.
For each $i,j$ define an $m$-dimensional vector
\[\alpha_{ij}^{[s,t]}=\Big(c_{ij1}^{[s,t]},\dots, c_{ijm}^{[s,t]}\Big).\]
By multiplication rule (D) we have from \eqref{KC} that
\[\alpha_{ij}^{[s,t]}=\alpha_{ij}^{[s,\tau]}A^{[\tau]}(A^{[t]})^{-1}.\]
This can be rewritten as
\[\alpha_{ij}^{[s,t]}A^{[t]}=\alpha_{ij}^{[s,\tau]}A^{[\tau]}.\]
It follows from the last equality  that the vector $\alpha_{ij}^{[s,t]}A^{[t]}$
should not depend on $t$, i.e. there is a vector $\beta_{ij}^{(s)}=(\beta_{ijk}^{(s)})_{k=1}^m$ such that $\alpha_{ij}^{[s,t]}A^{[t]}=\beta_{ij}^{(s)}$.
Then
\[\alpha_{ij}^{[s,t]}=\beta_{ij}^{(s)}(A^{[t]})^{-1},\]
i.e.
\[c_{ijr}^{[s,t]}=\sum_{k=1}^m\beta_{ijk}^{(s)}b_{kr}^{[t]}.\]
By \eqref{cd} we should have
\[\sum_{j=1}^mc_{ijr}^{[s,t]}=\sum_{k=1}^m\sum_{j=1}^m\beta_{ijk}^{(s)}b_{kr}^{[t]}=(A^{[s]}(A^{[t]})^{-1})_{ir}=\sum_{k=1}^m a_{ik}^{[s]}b_{kr}^{[t]},\]
i.e.,
\[\sum_{k=1}^mb_{kr}^{[t]}\Big(\sum_{j=1}^m\beta_{ijk}^{(s)}- a_{ik}^{[s]}\Big)=0,\]
the last equality follows from the condition of theorem.
\end{proof}

\emph{Case (E).} Comparing (D) and (E) one can see that
these multiplications are very similar: in case (D) the right hand side of the equation \eqref{KC} has sum
over two indexes of the matrix $\mathcal M^{[\tau,t]}$, but in the case (E) such sum is with respect to matrix $\mathcal M^{[s,\tau]}$.
With respect to the time these multiplications interchange the role of $s$ and $t$, i.e. multiplication (D) gives a forward flow, but multiplication (E) corresponds to a backward flow.
Therefore one can show that Proposition~\ref{pD} is true in the case (E)  too.
Here we shall give an example of an FA of type (E),
which corresponds to the matrix \eqref{ut}.
 \begin{ex}\label{e10}
We shall construct a cubic matrix $\M^{[s,t]}$ for $m=2$ which satisfies \eqref{KC} in the case of the multiplication (E).
Find such a matrix $\M^{[s,t]}$ corresponding to  the matrix \eqref{ut}:
from \eqref{KC} we get\footnote{Compare with case (D) to see the interchange of $s$ and $t$.}
\begin{align*}
c_{1ij}^{[s,t]}& = \ \ \frac{\Psi(\tau)}{ 2\Psi(s)}c_{1ij}^{[\tau,t]}-\frac{\Psi(\tau)}{2\Psi(s)}c_{2ij}^{[\tau,t]}, \qquad   i,j=1,2, \\
c_{2ij}^{[s,t]} & =-\frac{\Psi(\tau)}{ 2\Psi(s)}c_{1ij}^{[\tau,t]}+\frac{\Psi(\tau)}{2\Psi(s)}c_{2ij}^{[\tau,t]}, \qquad i,j=1,2.
\end{align*}
Consequently $c_{1ij}^{[s,t]}=-c_{2ij}^{[s,t]}$. Therefore
\[c_{1ij}^{[s,t]}=\frac{\Psi(\tau)}{\Psi(s)}c_{1ij}^{[\tau,t]} \Leftrightarrow \Psi(s)c_{1ij}^{[s,t]}=\Psi(\tau)c_{1ij}^{[\tau,t]}.\]
Hence $\Psi(s)c_{1ij}^{[s,t]}$ does not depend on $s$, i.e. there exists a function $\gamma_{ij}(t)$ such that $\Psi(s)c_{1ij}^{[s,t]}=\gamma_{ij}(t)$. Thus
\begin{equation}\label{ece}
 c_{1ij}^{[s,t]}=\frac{\gamma_{ij}(t)}{ \Psi(s)}, \qquad  c_{2ij}^{[s,t]}=-\frac{\gamma_{ij}(t)}{ \Psi(s)}.
\end{equation}
We should have
\[c_{111}^{[s,t]}+c_{121}^{[s,t]}=-(c_{112}^{[s,t]}+c_{122}^{[s,t]})=-(c_{211}^{[s,t]}+c_{221}^{[s,t]})=c_{212}^{[s,t]}+c_{222}^{[s,t]}=\frac{\Psi(t)}{ 2\Psi(s)}.\]
By using \eqref{ece}, we get from this equalities that
\[\gamma_{11}(t)+\gamma_{21}(t)=-(\gamma_{12}(t)+\gamma_{22}(t))=\frac{\Psi(t)}{2}.\]
Now the matrix \eqref{edd} has the following form:
\begin{equation}\label{ed2}
\M^{[s,t]}=\frac{1}{\Psi(s)}
\begin{pmatrix}
\gamma_{11}(t) &\gamma_{12}(t)& \vline &-\gamma_{11}(t) &-\gamma_{12}(t) \\[3mm]
\frac{\Psi(t)}{2}-\gamma_{11}(t) &-\gamma_{12}(t)-\frac{\Psi(t)}{2}& \vline &-\frac{\Psi(t)}{2}+\gamma_{11}(t) &\gamma_{12}(t)+\frac{\Psi(t)}{2}
\end{pmatrix},
\end{equation}
where $\Psi\ne 0$, $\gamma_{11}$ and $\gamma_{12}$ are arbitrary functions.
 This is a $2\times2\times 2$ cubic matrix which satisfies \eqref{KC} and has three parameter functions.
Thus we proved the following:
\begin{pro} For any fixed $\Psi\ne 0$, $\gamma_{11}$ and $\gamma_{12}$ the matrix \eqref{ed2} generates an FA of type (E).
\end{pro}
\end{ex}

The following theorem is an analogue of Theorem~\ref{tA} for an FA of type (E):
\begin{thm}\label{tE}
 Let $\{A^{[t]}=(a_{ij}^{[t]}), \,t\geq 0\}$ be a family of
invertible (for all $t$), $n\times n$ square matrices and let $(A^{[t]})^{-1}=(b_{ij}^{[t]})$ denote the inverse of $A^{[t]}$.
Assume that $\gamma_{ijk}^{(t)}$, $i,j,k=1, \dots,m$,  are arbitrary functions such that
\[\sum_{j=1}^m\gamma_{ijk}^{(t)}=b_{ik}^{[t]}, \qquad \text{for any} \ \ i,k \ \ \text{and} \ \ t.\]
Then the cubic matrix
\[\M^{[s,t]}=\left(\sum_{k=1}^ma_{ik}^{[s]}\gamma_{kjr}^{(t)}\right)_{i,j,r=1}^m\]
generates an FA of type (E).
\end{thm}
\begin{proof}
 The proof is similar to the proof of Theorem~\ref{tA}. Here instead of $\alpha_{ij}^{[s,t]}$ one considers
\[v_{jr}^{[s,t]}=\left(c_{1jr}^{[s,t]},\dots, c_{mjr}^{[s,t]}\right)^T.\]
By the multiplication rule (E) we have from \eqref{KC} that
\[v_{jr}^{[s,t]}=A^{[s]}(A^{[\tau]})^{-1}v^{[\tau,t]}_{jr}.\]
Now repeating similar calculations in the proof of Theorem~\ref{tA}, we obtain the result.
\end{proof}

\section{Time dependent dynamics of algebraic properties in an FA}\label{S:dynamics}

By our definitions, an FA is a continuous-time dynamical
system which in a fixed pair of time $(s,t)$ is an algebra.
Therefore properties of the algebras in an FA depend on time.

\begin{defn}
Assume an FA, $\mathcal A^{[s,t]}$, has a property, say $P$,
at a pair of times $(s_0,t_0)$; we say that the FA has $P$ property
transition if there is a pair $(s,t)\ne (s_0,t_0)$ at which the FA
has no the property $P$.
\end{defn}

Denote
\begin{align*}
\mathcal T & =\{(s,t): 0\leq s\leq t\};\\
 \mathcal T_P & =\{(s,t)\in \mathcal T: \mathcal A^{[s,t]} \ \ \text{has property} \  P \};\\
\mathcal T_P^0 & = \mathcal T\setminus \mathcal T_P=\{(s,t)\in \mathcal T: \mathcal A^{[s,t]} \ \ \text{has not property} \ P \}.
\end{align*}

\begin{defn}
We call the set:
\begin{itemize}
\item[-] $\mathcal T_P$: the duration of the property $P$;

\item[-] $\mathcal T_P^0$: the lost duration of the property $P$;

\item[-] The partition $\{\mathcal T_P, \mathcal T^0_P\}$ of the set
$\mathcal T$ is called $P$ property diagram.
\end{itemize}
\end{defn}

\subsection{Dynamics of a baric property.}
The following theorem gives dynamical baric properties:

\begin{thm}\hfill
 \begin{itemize}
\item[(i)] Let $\{\mathcal A^{[s,t]}, \ \ 0\leq s\leq t\}$ be a SFA.
 Then the algebra $\mathcal A^{[s,t]}$ is baric for any $(s,t)\in \mathcal T$.

 \item[(ii)] An algebra from an FA of type (C) may or may not be a baric algebra.

 \item[(iii)] Let $\{\mathcal A^{[s,t]}, \ \ 0\leq s<t\}$ be a flow of $m$-dimensional algebras of type (D) or (E). Then $\mathcal A^{[s,t]}$ is not baric if $m\geq 2$, i.e. $\mathcal T_{baric}=\emptyset$ if $m\geq 2$.
 \end{itemize}
 \end{thm}
\begin{proof} We use Proposition~\ref{pb}:

(i) By definition of a SFA the corresponding matrix $\M^{[s,t]}$ is stochastic, i.e. the condition of Proposition~\ref{pb} is satisfied.

(ii) In the case of multiplication (C), if the matrix $\mathcal M^{[s,t]}=(c_{ijk}^{[s,t]})$ satisfies
conditions
\begin{equation}\label{bc}
c_{ijk}^{[s,t]}=c_{jik}^{[s,t]}, \qquad \sum_{k=1}^mc_{ijk}^{[s,t]}=1, \quad  \text{for some} \ \ (s,t),
\end{equation}
then the corresponding $\mathcal A^{[s,t]}$ is baric, otherwise it is not baric.

(iii) We should check \eqref{bc} for the multiplication (D) (the case (E) is similar).
Assume $\mathcal A^{[s,t]}$ is baric for any $(s,t)\in \mathcal T$.  Then from \eqref{KC} we have
\[c_{ijr}^{[s,t]}=\sum_{k=1}^mc_{ijk}^{[s,\tau]}\sum_{n=1}^mc_{knr}^{[\tau,t]},\]
and, by using \eqref{bc}
for any $(s,t)$ we get
\[1=\sum_{r=1}^mc_{ijr}^{[s,t]}=\sum_{k=1}^mc_{ijk}^{[s,\tau]}\sum_{n=1}^m\sum_{r=1}^mc_{knr}^{[\tau,t]}=
m.\]
This is a contradiction to the assumption that $m\geq 2$.
\end{proof}

\begin{rk}
Since the matrix of an FA is not stochastic in general, using  the above mentioned examples of FAs
one can give several baric property diagrams.
  For example, in case (C) one can take 
 matrices mentioned in Example \ref{e7} which satisfy condition (\ref{bc}). 
  But in cases (D) and (E) baric FAs arise only for $m=1$.
\end{rk}

\subsection{Limit algebras of flows} As it was mentioned above
each flow of algebras can be considered as a continuous time dynamical system such that 
in each fixed pair $(s,t)$ of time
its state is a finite-dimensional algebra.
The main problem in such dynamical systems is to know what
will be the limit algebra (if it exists):
\[\lim_{t-s\to+\infty}\mathcal A^{[s,t]}=\mathcal A.\]
In case the limit does not exist then one asks about possible limit algebras
(see \cite[Proposition 2.7 ]{CLR} for an interesting example).

\emph{Case Example~\ref{e2}.} Simple calculations show that
the limit algebra (depending on parameter $\epsilon$ of the example) is $\mathcal A_\epsilon=\langle e_1, e_2\rangle$  with multiplication table:
\[
\begin{cases}
e_1^2=e_1e_2=e_2e_1=e_2^2=e_2, & \text{if} \ \ \epsilon\in (0,\frac{1}{2}],\\
e_1^2=e_1e_2=e_2e_1=e_2^2=xe_1+(1-x)e_2, & \text{if} \ \ \epsilon=0.
\end{cases}
\]

\emph{Case Example~\ref{e3}.} In this case the limit algebra (independently on parameter $\epsilon$ of the example) is
$\mathcal A=\langle e_1, e_2\rangle$  with multiplication table:
\[ e_1^2=e_1e_2=e_2e_1=e_2^2=e_2.\]

\emph{Case Example~\ref{e4}.} The limit algebra (depending on parameter $\epsilon$ and initial point $(x_1,x_2,1-x_1-x_2)$ of the example) is
$\mathcal A_\epsilon=\langle e_1, e_2, e_3\rangle$ with multiplication table:
\[e_ie_j=\begin{cases}
e_3, & \text{if} \ \ \epsilon\in [0,\frac{1}{2}),\\
x_1e_1+x_2e_2+(1-x_1-x_2)e_3,&  \text{if} \ \ \epsilon=\frac{1}{2},
\end{cases} \ \ \text{for any} \qquad i,j=1,2,3.
\]
It is easy to see that the limit algebra of Example~\ref{e5} coincides with the limit algebra of Example~\ref{e4}.
Thus we note that though these examples present SFAs of different types their limit algebras coincide.

\emph{Case Example~\ref{e6}.} The limit algebra depends on the given stochastic vector $a(t)$. If $\displaystyle \lim_{t\to\infty} a(t)$ exists then
the limit algebra also exists. But one can choose $a(t)$ such that the limit does not exist. For example take
$a(t)=(\sin^2(t), \cos^2(t),0,\dots,0)$ then the limit does not exist and the set of all possible limit algebras is
a continuum set. We do this point more clear in the case of Example~\ref{e9}.

\emph{Case Example~\ref{e9}.} It is clear that the limit algebra does not exist in Example~\ref{e9} either. Here we shall show that
the set of all limit algebras in discrete time is dense: Consider discrete time $n$, $n\in \mathbb N$, and the FA
$\{\mathcal A^{[n]}, n\in \mathbb N\}$ given by the matrix \eqref{ed3}.
Denote by $\mathcal A_a$, $a\in [-1,1]$,  the algebra which is given by  the matrix of structural constants \eqref{ed3}
with $\sin(\hat t)=a$.

\begin{pro}
 The discrete time FA $\mathcal A^{[n]}$, $n\in \mathbb N$, is dense in
the set $\{\mathcal A_a, \, a\in [-1,1]\}$ of algebras,
i.e. for an arbitrary algebra $\mathcal A_a$ there exists a
sequence $\{n_k\}_{k=1,2,\dots}$ of natural numbers such that
$\displaystyle \lim_{k\to\infty}\mathcal A^{[n_k]}=\mathcal A_a$.
\end{pro}
\begin{proof}
 It is known that the sequences $\{\sin n\}$ and $\{\cos
n\}, n\in \mathbb N$, are dense in $[-1,1]$.
Hence for any $a\in [-1,1]$ there is a sequence
$\{n_k\}_{k=1,2,\dots}$ of natural numbers such that
$\displaystyle  \lim_{k\to\infty}\sin(n_k)=a$. The same sequence can be used to
get $\displaystyle  \lim_{k\to\infty}\mathcal A^{[n_k]}=\mathcal A_a$.
\end{proof}

Since all other above-constructed FAs depend on arbitrary functions ($\psi, \kappa, \gamma, \dots$) their limit algebras
depend on these functions. One can control their limit by choosing these parameter functions.

\subsection{Dynamics of commutativity.} It is clear that a cubic matrix $(c_{ijk})$ generates a commutative algebra iff $c_{ijk}=c_{jik}$.
The algebras in FAs given by Examples~\ref{e2}--\ref{e6} are commutative for any time $(s,t)\in \mathcal T$.

In the case of Example~\ref{e7} an algebra $\mathcal A^{[s,t]}$ of the FA of type (C) is commutative iff $a_{2r}^{[s,t]}=\bar a_{1r}^{[s,t]}$.
For example, take for the matrix $(a_{ij}^{[s,t]})$ the matrix \eqref{defo} and for $(\bar a_{ij}^{[s,t]})$ its transpose.
Then as a particular case of Example~\ref{e7} the following cubic matrix generates an FA of type (C):
\[
\M^{[s,t]}=\begin{pmatrix}
\cos(t-s)&\sin(t-s)&\vline&-\sin(t-s)&\cos(t-s)\\[2mm]
\cos(t-s)&-\sin(t-s)&\vline&\sin(t-s)&\cos(t-s)
\end{pmatrix}.
\]
The algebra $\mathcal A^{[s,t]}$ of this FA is commutative
iff $\cos(t-s)=-\sin(t-s)$, i.e. $t=s+\frac{3\pi}{4}+\pi n$, $n=0,1,2,\dots$.
Since the FA is a time-homogenous, we can consider its dynamics with respect to
one time parameter $\hat t=t-s$. What was shown above is that the commutativity
of the algebras in this FA holds only at values $\frac{3\pi}{4}+\pi n$, $n=0,1,2,\dots$ of the time $\hat t$.
Thus dynamics of the commutativity is as follows: starting from $\hat t=0$ one first meets commutativity at $\hat t=\frac{3\pi}{4}$,
then the commutativity appears only after each adding time $\pi$, i.e with period $\pi$.
Therefore this is an example of almost non-commutative (with respect to Lebesgue measure on $\mathcal T$)
FA of type (C).

In the case of Example~\ref{e8} one can see that the duration of the commutativity is
the set
\[\{(s,t)\in \mathcal T: \kappa_{21}(s)=\frac{1}{2\Psi(s)}-\kappa_{11}(s)\}.\]
Since $\kappa_{11}$, $\kappa_{21}$ and $\Psi$ are arbitrary one can choose these functions
to control this set (the case of Example~\ref{e10} is similar).

The following proposition shows that algebras of  the FA in Example~\ref{e9}
will stay always commutative (or always non-commutative):

\begin{pro} The FA corresponding to Example~\ref{e9} (i.e. to  the matrix \eqref{ed3})
is always commutative iff $a=1-c$ and $b=-d$, otherwise they are always non-commutative.
\end{pro}
\begin{proof} Commutativity of the algebra $\mathcal A^{[\hat t]}$ is reduced to the system
\begin{equation}\label{c9}
 \begin{aligned}
 (1-a-c)\cos\hat t+(b+d)\sin\hat t & =0,  \\
  -(1-a-c)\sin\hat t+(b+d)\cos\hat t & =0.
  \end{aligned}
\end{equation}

Since $\det \begin{pmatrix*}[r]
\cos\hat t&\sin\hat t\\
-\sin\hat t&\cos\hat t
  \end{pmatrix*}=1$ for any $\hat t\geq 0$, the system \eqref{c9} holds only in the case  $a=1-c$ and $b=-d$.
  \end{proof}
The following proposition is a corollary of Theorem~\ref{tA} and Theorem~\ref{tE}:
\begin{pro} The commutativity duration for the FA given in Theorem~\ref{tA} (resp. Theorem~\ref{tE}) is
\[\{(s,t)\in \mathcal T: \ \ \beta^{(s)}_{ijk}=\beta^{(s)}_{jik}, \ \ \text{for any} \ \  i,j,k=1,\dots,m\}.\]
\[\Big(\text{resp.} \  \{(s,t)\in \mathcal T: \ \ \gamma^{(t)}_{ijk}=\gamma^{(t)}_{jik},  \ \ \text{for any} \ \  i,j,k=1,\dots,m\}\Big).\]
\end{pro}
\begin{proof} From the cubic matrix $\M^{[s,t]}$ mentioned in Theorem~\ref{tA} we get that the commutativity
of the corresponding algebra $\mathcal A^{[s,t]}$ holds for $(s,t)$ such that
\[\sum_{k=1}^m\left[\beta_{ijk}^{(s)}-\beta_{jik}^{(s)}\right]b^{(t)}_{kr}=0, \ \ \text{for all} \ i,j.\]
Since for any $t$, by the  condition of Theorem~\ref{tA}, we have $\det((A^{[t]})^{-1})=\det((b^{(t)}_{kr})_{k,r=1}^m)\ne 0$, and from the last system of equations we get
$\beta_{ijk}^{(s)}=\beta_{jik}^{(s)}$ for any $i,j$ and $k$.
\end{proof}

\subsection{Duration to be an evolution algebra}

 Let $(E,\cdot)$ be an algebra over a field $K$. If it admits a
basis $e_1,e_2,\dots$, such that $e_i\cdot e_j=0$, if $i\ne j$ and $
e_i\cdot e_i=\sum_{k}a_{ik}e_k$, for any $i$, then this algebra is
called an evolution algebra (EA), see \cite{t} for the theory of such algebras.

The following proposition is for stochastic FAs:
\begin{pro} Any SFA $\{\mathcal A^{[s,t]}: s,t\in \mathbb R\}$ does not contain an evolution algebra, i.e.
$\mathcal A^{[s,t]}\ne$ EA for any $(s,t)\in \mathcal T$.
\end{pro}
\begin{proof} Since in a stochastic flow of algebras we have the condition that
\[\sum_{k=1}^m P_{ij,k}^{[s,t]}=1, \ \text{for all} \  i,j \ \ \text{and} \ \ \text{for all} \ (s,t)\in \mathcal T,\]
for each $(s,t)$ and $i,j$, there exists $k_0=k_0(i,j,s,t)$ such that $P_{ij,k_0}^{[s,t]}>0$.
Consequently, $e_ie_j\ne 0$ for each $i\ne j$, i.e. the algebra is not an EA.
\end{proof}
We note that a non-stochastic FA may contain EAs:
in the case of Example~\ref{e8} one can see that the duration to be an EA in the corresponding FA is
the set
\[\mathcal E=\{(s,t)\in \mathcal T: \kappa_{21}(s)=\frac{1}{2\Psi(s)}-\kappa_{11}(s)=0\}.\]
Since $\kappa_{11}$, $\kappa_{21}$ and $\Psi$ are arbitrary for each given subset $\widetilde{\mathcal T}\subset \mathcal T$
one can choose these functions in such a way that $\mathcal E=\widetilde{\mathcal T}$ (the case of Example~\ref{e10} is similar).

\subsection{Dynamics of associativity} It is known \cite{Ja} that a finite-dimensional algebra $A$
with a matrix of structural constants $(c_{ijk})_{i,j,k=1}^n$ is associative iff
\begin{equation}\label{ass}
\sum_{r=1}^nc_{ijr}c_{rkl}=\sum_{r=1}^nc_{irl}c_{jkr}, \qquad \text{for all} \ i,j,k,l.
\end{equation}
In this subsection we check the associativity condition for FAs of Theorem~\ref{tA} and \ref{tE}.
\begin{thm}
 The duration of associativity for the FA $\{\mathcal A^{[s,t]}\}$ given in Theorem~\ref{tA} (resp. Theorem~\ref{tE}) is
\[\Big\{(s,t)\in \mathcal T: \sum_{p,r}\left(\beta_{ijp}^{(s)}\beta_{rkq}^{(s)}-\beta_{irq}^{(s)}\beta_{jkp}^{(s)}\right)b^{[t]}_{pr}=0, \ \  \text{for all} \ i,j, k,q\Big\}.\]
\[\Big( \text{resp.} \ \ \Big\{(s,t)\in \mathcal T: \sum_{r,q}\Big(a_{rq}^{[s]}\gamma_{pjr}^{(t)}\gamma_{qkl}^{(t)}-a_{jq}^{[s]}\gamma_{prl}^{(t)}\gamma_{qkr}^{(t)}\Big)=0, \ \ \text{for all} \ p,j,k,l\Big\}\Big).\]
\end{thm}
\begin{proof} For $\mathcal A^{[s,t]}$ the condition \eqref{ass} can be written as
\[\sum_{p,r,q}\beta_{ijp}^{(s)}\beta_{rkq}^{(s)}b^{[t]}_{pr}b^{[t]}_{ql}=\sum_{p,r,q}\beta_{irq}^{(s)}\beta_{jkp}^{(s)}b^{[t]}_{pr}b^{[t]}_{ql}, \qquad \text{for all} \  i,j,k,l.\]
Rewrite it in the following form:
\begin{equation}\label{X}
\sum_q X_{ijk,q}b^{[t]}_{ql}=0, \qquad \text{for all} \  i,j,k,l,
\end{equation}
where \[X_{ijk,q}(s,t)=\sum_{p,r}\left(\beta_{ijp}^{(s)}\beta_{rkq}^{(s)}-\beta_{irq}^{(s)}\beta_{jkp}^{(s)}\right)b^{[t]}_{pr}.\]
Since $\det((A^{[t]})^{-1})\ne 0$ from \eqref{X} we get $X_{ijk,q}(s,t)\equiv 0$, i.e.
\[\sum_{p,r}\left(\beta_{ijp}^{(s)}\beta_{rkq}^{(s)}-\beta_{irq}^{(s)}\beta_{jkp}^{(s)}\right)b^{[t]}_{pr}=0.\]
\end{proof}

\begin{rk} An interesting problem is to construct an FA which contains (depending on time)
as many known finite-dimensional algebras as possible.
\end{rk}
\section*{ Acknowledgements}

 This work was partially supported by Ministerio de Econom\'{\i}a y Competitividad (Spain), grant MTM2013-43687-P
  (European FEDER support included), by Xunta de Galicia, grant GRC2013-045 (European FEDER support
included) and by Kazakhstan Ministry of Education and Science, grant 0828/GF4: ``Algebras, close to Lie:
cohomologies, identities and deformations''. We thank the referee for careful reading of the manuscript and for useful suggestions; 
in particular, for picking an error in the earlier version of Example 8 and Example 10.

{}
\end{document}